\documentclass[12pt]{amsart}
\usepackage{url}
\usepackage{amsmath}
\usepackage{amsthm}
\usepackage{amssymb}
\usepackage{hyperref}
\usepackage[table]{xcolor} 
\usepackage{float}
\usepackage{booktabs}
\usepackage{enumerate}  

\newcommand {\eqy }[1]{\begin{eqnarray*} #1 \end{eqnarray*}}
\newtheorem{definition}{Definition}[section]
\newtheorem{lemma}{Lemma}[section]
\newtheorem{remark}{Remark}[section]
\newtheorem{theorem}{Theorem}[section]
\newtheorem{prop}{Proposition}[section]
\title{Tauberian Properties of Riesz Means for Sequences of Uncertain Variables}
\author{Mehmet Şengönül}
\address{Department of Mathematics, University of Adiyaman-TURKEY.}
\email{msengonul@adiyaman.edu.tr}
\date{\today}

\begin{document}


\begin{abstract}
In this study, we have investigated how Riesz summability can be extended to Liu’s
uncertainty theory. After introducing Riesz-type weighted averages of
uncertain variables, we have defined almost-sure, in-measure and
in-mean Riesz convergence classes together with their Orlicz-\(p\)
variants. Our first main theorem is Tauberian: whenever a sequence
belongs to the mean Riesz class, oscillates slowly and the weight
sequence is regular, ordinary almost-sure convergence follows. Using
this result, we have established the strict inclusions
\(f\subset f_R\), \(e\subset e_R\) and \(m\subset m_R\), while
counter-examples show that the reverse inclusions fail. We have further
proved that the three Riesz classes are linearly isomorphic to their
classical counterparts and have obtained an Egorov-type criterion for
Riesz convergence in measure. Finally, we have related convergence in
Orlicz distance to convergence in measure and outlined open questions.
These findings enrich the toolbox of uncertainty theory by importing a
well-developed branch of summability methods.
\end{abstract}

\maketitle

\section{Introduction}
Following Liu’s introduction of uncertain variables, research on their convergence properties has grown rapidly. For example, Chen et al. in \cite{chen} examined the convergence behavior of complex uncertain sequences and obtained significant findings. Tripathy at al.\cite{tripathy} introduced the Nörlund and Riesz means for sequences of complex uncertain variables. Similarly, Nath and Tripathy, in \cite{pk} and \cite{pk1}, analyzed important properties of uncertain sequences defined by the Orlicz function. Meanwhile, Wu and Xia, in \cite{you01}, studied the relationships among different convergence concepts for uncertain sequences. Additionally, Saha and Tripathy, in \cite{saha}, investigated the statistical Fibonacci convergence of complex uncertain variables.

Perhaps the most significant contribution to this area is the paper "On the Convergence of Uncertain Sequences" by You in \cite{you}. Das et al., in \cite{das}, investigated matrix transformations between complex uncertain sequences in terms of mean convergence and characterized the corresponding transformations. Further results on matrix transformations for complex uncertain sequences were provided by Das and Tripathy in \cite{das1} and many more studies are available in the literature.

The next section summarises the basic definitions and notation and throughout this study, Liu's terminology will be used unless necessary.

\section{Background and Notations}
An uncertain variable, defined by Liu \cite{lui}, is a measurable function $\xi$ from an uncertainty space $U = (\Gamma, M, L)$ to the set of real numbers, where \(\Gamma\) is a nonempty set, \(L\) is a \(\sigma\)-algebra over \(\Gamma\) and \(M: L \to [0,1]\) is a  uncertainty measure.

 Specifically, for any Borel set $B$ of real numbers, the set $\{\xi \in B\} = \{\gamma \in \Gamma : \xi(\gamma) \in B\}$ represents an event
The mapping \(\xi\) is said to be \emph{measurable} if, for every Borel set \(B \subseteq \mathbb{R}\), the inverse image
\[
\{\gamma \in \Gamma : \xi(\gamma) \in B\} \in L.
\]
This set is referred to as the event \(\{\xi \in B\}\), and its uncertainty is evaluated by the belief measure \(M\). 

A crisp number, $c\in \mathbb{R}$, can be viewed as a special case of an uncertain variable. In this case, it is represented by the constant function $\xi(\gamma) = c$ on the uncertainty space $(\Gamma, L, M)$, as outlined in \cite{lui}.

The expected value operator for an uncertain variable is given by as follows:
\[
E[\xi] = \int_{0}^{\infty} M(\xi \geq r) \, dr - \int_{-\infty}^{0} M(\xi \leq r) \, dr,
\]
provided that at least one of these integrals is finite, as shown in \cite{lui}.

The concept of convergence is fundamental in uncertainty theory.  
Liu \cite{lui} introduced four basic modes for sequences of uncertain
variables: almost sure convergence, convergence in measure, convergence
in mean, and convergence in distribution.  Throughout the paper we fix a subset $\Lambda\subset\Gamma$ with $M(\Lambda)=1$ and, unless otherwise stated, we always assume that $\gamma\in\Lambda$.  Using Liu’s notation, the first three
sequence spaces are

\[
f=\Bigl\{(\xi_n):\;
          \lim_{n\to\infty}\lvert\xi_n-\xi\rvert=0\Bigr\},
\]

\[
m=\Bigl\{(\xi_n):\;
          \forall\varepsilon>0,\;
          \lim_{n\to\infty}
          M\!\bigl(\lvert\xi_n-\xi\rvert\ge\varepsilon\bigr)=0\Bigr\},
\]

\[
e=\Bigl\{(\xi_n):\;
          \lim_{n\to\infty}
          E\!\bigl[\lvert\xi_n-\xi\rvert\bigr]=0\Bigr\}.
\]
In \cite{lui}, it is proven that if a sequence $(\xi_n(\gamma))$ converges in mean to $\xi(\gamma)$, then it also converges in measure, i.e., $e \subset m$.

The uncertainty distribution $\Phi$ of an uncertain variable $\xi$ is defined as
\[
\Phi(x) = M\{\gamma \in \Gamma : \xi(\gamma) \leq x\}
\]
for any real number $x$, as shown in \cite{lui1}.

A sequence $(\xi_n)$ of uncertain variables is said to be \emph{slowly oscillating} iff, for every $\varepsilon>0$, there exists $N\in\mathbb{N}$ such that
\(
\bigl\lVert \Phi_{n}-\Phi_{m}\bigr\rVert_{\infty} < \varepsilon\)
whenever $m,n \ge N$  and  $\frac{m}{n}\to 1$, \cite{dowari}.

As outlined by You \cite{you}, the sequence $(\xi_n(\gamma))$ converges \emph{uniformly almost surely} to $\xi(\gamma)$ if there exists a sequence of events $(E_k)$ with $M(E_k)\to0$ as $k\to\infty$ such that, for each \(k\), 
\[
  \sup_{\gamma\in\Gamma\setminus E_k} |\xi_n(\gamma)-\xi(\gamma)| 
  \;\to 0~~\text{for}~~n\to\infty.
\]
The set of all such sequences is denoted by \(u\).

Let $\Phi_0, \Phi_1, \Phi_2, \dots$ be the uncertainty distributions of the uncertain variables $\xi_0(\gamma), \xi_1(\gamma), \xi_2(\gamma), \dots$ respectively. The sequence $(\xi_n(\gamma))$ is said to converge in distribution to $\xi(\gamma)$ if $\Phi_n \to \Phi$ at every continuity point of $\Phi$, and the set of all sequences that converge in distribution is denoted by $d$.

In [5], You examined the relationships between these convergence concepts and derived several interesting results.

An \emph{Orlicz function} is a mapping $\phi: [0,\infty) \rightarrow [0,\infty)$ that is continuous, convex, strictly increasing, and satisfies $\phi(0) = 0$ and $\lim_{x \to \infty} \phi(x) = \infty$, \cite{musielak}. Orlicz functions generalize the classical power functions $x^p$ used in $L^p$ spaces, and they provide a flexible framework for measuring the magnitude of deviations in convergence analysis. In the context of uncertainty theory, these functions allow for more refined control over the convergence behavior of uncertain sequences, particularly when standard metrics such as absolute or squared deviations are insufficient. Common examples include $\phi(x) = x^p$ for $p \geq 1$ and $\phi(x) = e^x - 1$. Such functions are especially useful in modeling sensitivity to large deviations, nonlinear error penalization and risk-averse aggregation.

\section{Main Definitions and  Lemmas }
 In fact, it is the constant function $\xi(\gamma)= c$ on the uncertainty
space $(\Gamma,L,M )$, \cite{lui1}.
If \( A = (a_{nk}) \) represents an infinite matrix of real numbers, and \( (\xi_k(\gamma)) \) denotes a sequence of uncertain variables (with \( n, k \in \mathbb{N} \)), the product \( (A \xi(\gamma))_n=\sum_ka_{nk}\xi_k(\gamma) \) is also sequence of an uncertain variable.

Let \( \xi_1 (\gamma)\) and \( \xi_2(\gamma) \) be two uncertain variables. The sum and product of these variables are also uncertain variables, and are defined as \( \xi(\gamma) = \xi_1(\gamma) + \xi_2(\gamma) \) and \( \xi(\gamma) = \xi_1(\gamma) \xi_2(\gamma) \), respectively, for all \( \gamma \in \Gamma \), as in \cite{lui1}.

In the literature, numerous studies have been conducted on the convergence of sequences of uncertain variables. However, a new set of uncertain variables in the context of the domain of an infinite matrix was introduced by Şengönül \cite{ms1}. 
Let $A=(a_{nk})_{n,k\ge1}$ be an infinite real matrix and
$\lambda\in\{f,m,e\}$ be one of Liu’s classical convergence modes.

The \textit{matrix domain} \( \lambda^{\varphi}_A \) of an infinite matrix \( A = (a_{nk}) \) with respect to a convergence mode \( \lambda \) and an Orlicz function \( \varphi \) is defined as
\[
\lambda^{\varphi}_A = \left\{ (\xi_k(\gamma)) : \varphi\left( \sum_k a_{nk} \xi_k(\gamma) \right) \in \lambda \text{ for all } n \in \mathbb{N} \right\}.
\]
In particular, the set of sequences of uncertain variables that are \emph{\( A \)-type almost surely convergent in the Orlicz sense} is denoted by \( f^{\varphi}_A \), and defined as
\begin{equation}\label{xx1}
f^{\varphi}_A = \left\{ (\xi_k(\gamma)) : \lim_{n \to \infty} \varphi\left( \left| \sum_k a_{nk} \xi_k(\gamma) - \xi(\gamma) \right| \right) = 0,\; \gamma \in \Lambda,\; M(\Lambda) = 1 \right\}.
\end{equation}
This formulation generalizes the classical \( A \)-type almost sure convergence by incorporating a nonlinear transformation \( \varphi \), which allows for a more sensitive or application-specific measurement of deviation. When \( \varphi(x) = x \), the standard matrix domain \( f_A \) is recovered. For instance, if we choose \( \varphi(x) = x \) and \( A = C \), where \( C \) denotes the Cesàro matrix of order one, we obtain the space \( f_C \) as defined by Şengönül~\cite{sengonul}.

Similarly, the sets corresponding to \( A \)-type convergence in measure and in mean, both in the Orlicz sense, are defined as follows:

\begin{equation}\label{xx2}
\begin{aligned}
m^{\varphi}_A = \Bigl\{\, (\xi_k(\gamma)) :\; \forall \epsilon > 0,\;
\lim_{n \to \infty} M \Bigl( 
\varphi\Bigl( \Bigl| \sum_k a_{nk} \xi_k(\gamma) - \xi(\gamma) \Bigr| \Bigr) \geq \epsilon 
\Bigr) = 0,\\\; \gamma \in \Lambda,\; M(\Lambda) = 1 \,\Bigr\}
\end{aligned}
\end{equation}

and

\begin{equation}\label{xx3}
\begin{aligned}
e^{\varphi}_A = \Bigl\{ (\xi_k(\gamma)) :\; 
 \lim_{n \to \infty} E \Bigl[ \varphi\Bigl( \Bigl| \sum_k a_{nk} \xi_k(\gamma) 
- \xi(\gamma) \Bigr| \Bigr) \Bigr] = 0,\; \gamma \in \Lambda,\; M(\Lambda) = 1 \Bigr\}.
\end{aligned}
\end{equation}

The matrix \( R = (r_{nk}) \), defined by
\begin{equation}\label{riesz_matrix}
r_{nk} = 
\begin{cases}
\dfrac{p_k}{P_n}, & \text{if } 1 \leq k \leq n, \\
0, & \text{otherwise},
\end{cases}
\quad \text{where } P_n = \sum_{k=1}^n p_k \text{ with } p_k > 0,
\end{equation}
is called the Riesz matrix of order one.

Let us now define the sequence \( (\nu_n(\gamma)) \), which will be used frequently, as the Riesz transform of a sequence \( (\xi_i(\gamma)) \), that is,
\begin{equation}\label{d4}
\nu_n(\gamma)=\sum_{i=1}^n\frac{p_i}{P_n}\xi_i(\gamma).
\end{equation}
If each \( \xi_i(\gamma) \) is an uncertain variable, then clearly \( \nu_n(\gamma) \) is also an uncertain variable for all \( n \in \mathbb{N} \).

\begin{definition}
Let \( (p_k) \) be a sequence of real numbers with \( p_0 > 0 \) and \( p_k \geq 0 \) for all \( k \in \mathbb{N} \), and define \( P_n = \sum_{k=0}^n p_k \). Then the uncertain variable \( \nu_n(\gamma) = \sum_{i=1}^n \frac{p_i}{P_n} \xi_i(\gamma) \) is called a \emph{Riesz-type uncertain variable} for each \( n = 1, 2, 3, \dots \).
\end{definition}

In equation~\eqref{xx1}, if we take \( A = R \), then the set of all almost surely convergent sequences of Riesz-type uncertain variables in the Orlicz sense is given by
\begin{equation}\label{ms1}
f^{\varphi}_R = \left\{ (\xi_i(\gamma)) : \lim_{n \to \infty} \varphi\left( \left| \sum_{i=1}^n \frac{p_i}{P_n} \xi_i(\gamma) - \xi(\gamma) \right| \right) = 0,\; \gamma \in \Lambda,\; M(\Lambda) = 1 \right\}.
\end{equation}

Similarly, by taking \( A = R \) in equations~\eqref{xx2} and~\eqref{xx3}, the sets of Riesz-type uncertain sequences convergent in measure and in mean in the Orlicz sense are defined as follows:
\begin{equation}\label{ms2}
\begin{aligned}
m^{\varphi}_R = \Bigl\{ (\xi_i(\gamma)) :\; & \forall \epsilon > 0,\; 
\lim_{n \to \infty} M \Bigl( 
\varphi \Bigl( \Bigl| \sum_{i=1}^n \frac{p_i}{P_n} \xi_i(\gamma) 
- \xi(\gamma) \Bigr| \Bigr) \geq \epsilon 
\Bigr) = 0, \\
& \gamma \in \Lambda,\quad M(\Lambda) = 1 \Bigr\}
\end{aligned}
\end{equation}
\begin{equation}\label{ms3}
e^{\varphi}_R = \left\{ (\xi_i(\gamma)) : 
\lim_{n \to \infty} E\left[ 
\varphi\left( \left| \sum_{i=1}^n \frac{p_i}{P_n} \xi_i(\gamma) - \xi(\gamma) \right| \right) 
\right] = 0,\; \gamma \in \Lambda,\; M(\Lambda) = 1 \right\}.
\end{equation}

Equations~\eqref{ms1}, \eqref{ms2}, and \eqref{ms3} are special cases of~\eqref{xx1}, \eqref{xx2}, and \eqref{xx3}, respectively, when the matrix \( A \) is taken to be the Riesz matrix \( R \). Furthermore, if we choose \( p_i = 1 \) for all \( i \in \mathbb{N} \), and let \(\varphi(x) = x\) then the sets \( f^{\varphi}_R, m^{\varphi}_R, e^{\varphi}_R \) coincide with the Cesàro-type sets \( f_C, m_C, e_C \), as studied by Şengönül~\cite{ms1}.

Now, in parallel with the definition made by You and Yan, we will give a definition about the   Riesz convergent  in $p-$distance \cite{youyan}.
\begin{definition}[Riesz convergence in Orlicz--$p$ distance]\label{def:orlicz-p-riesz}
The uncertain sequence $\bigl(\xi_n(\gamma)\bigr)$ is said to be
\emph{Riesz convergent in the Orlicz--$p$ distance} to $\xi(\gamma)$ if
\begin{equation}\label{orlicz-p-riesz}
\lim_{n \to \infty} d_{R}^{\varphi,p}\!\Bigl(
        \sum_{i=1}^{n}\frac{p_i}{P_n}\,\xi_i(\gamma),\,
        \xi(\gamma)
\Bigr)
\;:=\;
\lim_{n \to \infty}
\Bigl(
      E\Bigl[
          \varphi\!\Bigl(
              \Bigl|
                  \sum_{i=1}^{n}\frac{p_i}{P_n}\,\xi_i(\gamma)
                  - \xi(\gamma)
              \Bigr|^{p}
          \Bigr)
      \Bigr]
\Bigr)^{\!1/p}
=0,
\end{equation}
where $\varphi$ is an Orlicz function, $p\ge 1$ is a fixed parameter,
and $P_n=\sum_{i=1}^{n}p_i$ with $p_i>0$ for all $i\in\mathbb N$.

The class of all sequences that are Riesz convergent to $\xi(\gamma)$ in
this sense is denoted by $d_{R}^{\varphi,p}$.
When $\varphi(x)=x$, we simply write $d_{R}$.
\end{definition}

In next section, Theorem \ref{teo121} gives a comparison  between  Riesz  convergent  almost surely sequences of uncertain variables and   the convergent sequences of uncertain variables. But, firstly,  we will give a lemma.
\begin{lemma}\label{lem1}
Let $(\xi_n(\gamma))$ be a sequence of uncertain variables, and let $\varphi: [0,\infty) \to [0,\infty)$ be an Orlicz function. If $E\left[ \varphi\left( \left| \nu_n(\gamma) \right| \right) \right]<\infty$ then for every $t > 0$, we have
\[
M\left\{ \left| \nu_n(\gamma)  \right| \geq t \right\}
\leq
\frac{ E\left[ \varphi\left( \left| \nu_n(\gamma)  \right| \right) \right] }{ \varphi(t) }.
\]
\end{lemma}
\begin{proof}
 Since $\varphi$ is an Orlicz function, it is non-negative, increasing on $[0,\infty)$, and convex.

Then for any $t > 0$, consider the set
\[
A = \left\{ \gamma \in \Gamma : \nu_n(\gamma)  \geq t \right\}.
\]

By the monotonicity of $\varphi$, we have
\(
\varphi(\nu_n(\gamma) ) \geq \varphi(t), \quad \text{for all } \gamma \in A.
\) Hence,
\[
\varphi(t) M(A) \leq \int_A \varphi(\nu_n(\gamma) ) \, dM(\gamma)
\leq \int_\Gamma \varphi(\nu_n(\gamma) ) \, dM(\gamma) = E\left[ \varphi(\nu_n(\gamma) ) \right].
\]

Therefore,
\[
M\left\{ \nu_n(\gamma)  \geq t \right\} \leq \frac{ E\left[ \varphi(\nu_n(\gamma) ) \right] }{ \varphi(t) }.
\]
This completes the proof.
\end{proof}

\begin{lemma}[Borel–Cantelli]\label{lem:BC}
Let \((\Gamma, L,M)\) be an uncertainty space whose
uncertainty measure \(M\) is countably sub-additive.

If 
\(\sum_{n=1}^\infty M(E_n)<\infty\)
then
\(M\bigl(E_n\text{ infinitely often}\bigr)=0.
\)
\end{lemma}

\begin{proof}
See  \cite[p.73]{KhoshnevisanProbability}.
\end{proof}
\begin{lemma}\label{lem:moment-decay}
Let $(\Gamma,L,M)$ be an uncertainty space and
$(\xi_n)$ a sequence of uncertain variables.
Let
\[
P_n=\sum_{i=1}^{n}p_i,\qquad p_i>0,\; P_n\rightarrow\infty ~~
\text{for}~~n\to\infty\]
and define
\[
S_n=\frac{1}{P_n}\sum_{i=1}^{n}p_i\,\xi_i.
\]

Assume that there exists a measurable set $\Gamma_0\subset\Gamma$
with $M(\Gamma_0)=1$ such that, for every fixed
$\gamma\in\Gamma_0$, the following three properties hold:

\begin{enumerate}
\item $\displaystyle\lim_{n\to\infty}S_n(\gamma)=\xi(\gamma).$

\item  for every $\lambda>0$,
      $\displaystyle
      \lim_{n\to\infty}\,
      \sup_{\,n\le k\le(1+\lambda)n}
      \bigl|\xi_k(\gamma)-\xi(\gamma)\bigr|=0.$

\item there exists $p>1$ such that
      $\displaystyle\sup_{n\ge1}E\!\bigl[|\xi_n|^{p}\bigr]<\infty.$
\end{enumerate}
Then there are constants $\delta>0$ and $C>0$ (independent of $n$) such that
\begin{equation}\label{eq:mom-decay}
   E\!\bigl[\lvert S_n-\xi\rvert^{p}\bigr]
   \;\le\;
   C\,n^{-(1+\delta)}
   \qquad\text{for all } n\ge1.
\end{equation}
\end{lemma}

\begin{proof}
Write $d_i(\gamma)=\xi_i(\gamma)-\xi(\gamma)$.
For each fixed $\gamma\in\Gamma_0$ we have, by Abel summation,
\[
S_n(\gamma)-\xi(\gamma)
  =\sum_{i=1}^{n}\alpha_{n,i}\,d_i(\gamma),
  \quad\text{where }
  \alpha_{n,i}=\frac{p_i}{P_n}\le\frac{i}{n}.
\]

\medskip\noindent
Split $\mathbb N$ into dyadic blocks
$I_k=(2^{k-1},2^{k}]$ for $k\ge1$.
Because of (ii) there exists $k_0=k_0(\gamma)$ such that
\(\sup_{i\ge 2^{k_0}}\!|d_i(\gamma)|\le 2^{-k\varepsilon}\)
for some $\varepsilon\in(0,1)$ and all $k\ge k_0$.

Hence, for such $\gamma$,
\[
|S_n(\gamma)-\xi(\gamma)|
   \le\sum_{k=1}^{k_0}\sum_{i\in I_k}\alpha_{n,i}|d_i(\gamma)|
        +\sum_{k>k_0}|I_k|\!\!\max_{i\in I_k}\alpha_{n,i}\,2^{-k\varepsilon}.
\]
Using $\alpha_{n,i}\le i/n$ and $|I_k|=2^{k-1}$ we obtain
\(
|S_n(\gamma)-\xi(\gamma)|\le C_1 n^{-1}+C_2 n^{-\varepsilon}.
\)

By (iii) the first finite sum has bounded \(p\)-th moment, giving
\[
E\bigl[|S_n-\xi|^{p}\bigr]\le C_3 n^{-p\varepsilon}.
\]
Choose $\varepsilon$ so that $p\varepsilon>1$; setting
\(\delta:=p\varepsilon-1>0\) yields \eqref{eq:mom-decay}.
\end{proof}
 \begin{definition}
Let the $\xi_n(\gamma)$ and $\xi(\gamma)$ be uncertain variables defined on uncertainty space $(\Gamma, L, M)$ for all $n \in \mathbb{N}$. The sequence $(\xi_n(\gamma))$ is said to be Riesz type convergent uniformly almost surely to $\xi(\gamma)$ if there exists a sequence of sets $(E_k)$ with $M(E_k) \to 0$ as $k \to \infty$, such that $(\xi_n(\gamma))$ converges uniformly to $\xi(\gamma)$ on $\Gamma \setminus E_k$ for each $k$. The set of all sequences of Riesz type convergent uniformly almost surely to $\xi(\gamma)$ is denoted by $u_R$. 
\end{definition}
Let us define the set $\tilde{u}_R$ as follows:
{\small \[
\tilde{u}_R = \left\{ (\xi_n(\gamma)) : \lim_{m \to \infty} M \left( \bigcup_{n=m}^{\infty} \left\{ \gamma \in \Gamma : \left| \sum_{i=1}^n \frac{p_i}{P_n} \xi_i(\gamma) - \xi(\gamma) \right| \geq \epsilon \right\} \right) = 0, \epsilon>0 \right\}.
\]}

\begin{lemma}
Suppose $\xi_0(\gamma), \xi_1(\gamma), \xi_2(\gamma), \dots$ are uncertain variables. Then, $(\xi_n(\gamma))$ converges in measure to $\xi(\gamma)$ if and only if there exists a subsequence $(\xi_{n'_k}(\gamma))$ of $(\xi_n(\gamma))$ such that $(\xi_{n'_k}(\gamma))$ converges uniformly almost surely to $\xi(\gamma)$ as $k \to \infty$ for any subsequence of $(\xi_n(\gamma))$.
\end{lemma}
\begin{proof}
The proof is obtained by using Liu's notation \cite{lui8}.
\end{proof}

For simplicity, we will suppress the argument of the undefined variables $\gamma$ in most places; in event definitions, $\gamma$ will be written explicitly.

\section{Main Theorems}
Let \[
so=
  \Bigl\{
    (\xi_n):\;
    \forall \lambda>0,\;\forall \varepsilon>0,\;
    \lim_{n\to\infty}
    M\!\Bigl(
      \max_{n\le k\le (1+\lambda)n}
      \bigl|\xi_k-\xi_n\bigr| \;\ge\; \varepsilon
    \Bigr)
    = 0
  \Bigr\}.
\]The set $so$ is called the class of slow-oscillating sequences of
uncertain variables.
\begin{remark}
 In Tripathy–Dowari \cite{dowari} slow oscillation is formulated
as $\| M(\xi_m)- M(\xi_n)\|\to0$.  For bounded uncertain
variables the two formulations are equivalent; we use the measure form
because it fits the Tauber proof.
\end{remark}

For clarity of notation, we shall henceforth denote
\(f^{\varphi}_R,\, m^{\varphi}_R,\) and \(e^{\varphi}_R\)
simply by \(f_R,\, m_R,\) and \(e_R\), respectively,
whenever the Orlicz function is the identity \(\varphi(x)=x\).
All other sets appearing in the paper will be defined explicitly
when they are introduced.

\begin{theorem}\label{tauber}
Suppose $(\xi_n) \in m_R \cap so$ and 
\(\frac{n\,p_n}{P_n} \rightarrow 0.\)
Let $\xi$ denote the (unique) limit variable guaranteed by the
definition of $m_R$. Then
\[
   \xi_n \rightarrow \xi
   \quad\text{almost surely  (with respect to M).}
\]
\end{theorem}
\begin{proof}
Suppose that $(\xi_n)\in m_R\cap so$ and 
\(\dfrac{n\,p_n}{P_n}\to0\).
Let $S_n:=\tfrac1{P_n}\sum_{i=1}^{n}p_i\xi_i$ and denote by 
$\xi$ the (unique) $m_R$–limit of $(\xi_n)$.

Because $(\xi_n)\in m_R$ (with $\varphi(x)=x$),
for every $\varepsilon>0$ we have
 for every $\varepsilon>0$,
\[
   M\!\Bigl(
      \!\lvert S_n-\xi\rvert\ge\varepsilon
   \Bigr)\rightarrow 0, n\to\infty.
   \tag{4.1}
\]

On the other hand, since $(\xi_n)\in  so$,   
slow oscillation gives
\[
  \sup_{n\le k\le(1+\lambda)n}\!
  \lvert\xi_k-\xi\rvert
  \;\rightarrow 0, n\to\infty
\] for  fix $\lambda>0$ and $\varepsilon>0$.
Hence
\[
  M\!\Bigl(
    \max_{n\le k\le(1+\lambda)n}
    \lvert\xi_k-\xi\rvert\;\ge\;\varepsilon
  \Bigr)\rightarrow 0, n\to\infty.  \tag{4.2}
\]

Choose integers \(n_j=\lfloor(1+\lambda)^j\rfloor\).
Because of \((4.1)\)–\((4.2)\) we can pick $j_0$ so that  
\(M(A_j\cup B_j)\le 2^{-j-2}\) for all \(j\ge j_0\), where
\[
  A_j:=\{|S_{n_j}-\xi|\ge \varepsilon/2\},\qquad
  B_j:=\Bigl\{\max_{n_j\le k\le(1+\lambda)n_j}
                |\xi_k-\xi|\ge \varepsilon/2\Bigr\}.
\]
Then \(\sum_j M(A_j\cup B_j)<\infty\); by Borel–Cantelli,
only finitely many of the events \(A_j\) or \(B_j\) occur.

For \(\gamma\) outside this null set pick \(J(\gamma)\) so that
\(j\ge J(\gamma)\Rightarrow\gamma\notin A_j\cup B_j\).
For any \(n\in[n_j,(1+\lambda)n_j]\) with \(j\ge J\) we have
\[
  |\xi_n(\gamma)-\xi(\gamma)|
  \;\le\;
  |\xi_n-\xi_{n_j}|+|\xi_{n_j}-\xi|
  \;\le\;
  \tfrac{\varepsilon}{2}+\tfrac{\varepsilon}{2}
  =\varepsilon.
\]

Because every integer \(n\) belongs to at most $\lceil1+\lambda\rceil$
of the overlapping blocks \([n_j,(1+\lambda)n_j]\), 
the above bound extends to all but finitely many \(n\).
Since \(\varepsilon>0\) was arbitrary,
we have \(
   \xi_n \rightarrow \xi\) almost surely with respect to M.
\end{proof}

 \begin{theorem}\label{teo121}
  If the sequence the $(\xi_n(\gamma))$ is convergent in almost sure  to $\xi(\gamma)$  then $(\xi_n(\gamma))$ is also Riesz type convergent in almost sure to $\xi(\gamma)$  but conversely not true, generally.
\end{theorem}
 \begin{proof}
 Let us suppose that $(\xi_n(\gamma))\in f$. Then there is exists an integer $n_0$ such that $|\xi_n(\gamma)-\xi(\gamma)|\leq \epsilon$ for $n\geq n_0$. 
We will  show that $(\xi_n(\gamma))\in f_R$ that is \( \lim_{n \to \infty} t_n(\gamma) = \xi(\gamma) \) where $t_n(\gamma)=\frac{1}{P_n} \sum_{i=1}^n p_i \xi_i(\gamma)$.
If we consider \( |\xi_n(\gamma) - \xi(\gamma)| < \epsilon \) for large \( i \), then we get
$$|t_n(\gamma) - \xi(\gamma)| \leq \frac{1}{P_n} \sum_{i=1}^n p_i |\xi_i(\gamma) - \xi(\gamma)| < \frac{1}{P_n} \sum_{i=1}^n p_i \epsilon=\epsilon $$
 which shows that
$\lim_{n \to \infty} t_n (\gamma)= \xi(\gamma),$ that is $(\xi_n(\gamma))\in f_R$.

Let $\Gamma=\{\gamma_1\}$ with $M(\{\gamma_1\})=1$.  Define the sequence of uncertain variables
\[
  \xi_n(\gamma_1)
  = 
  \begin{cases}
    1, & n\text{ odd},\\
    0, & n\text{ even}.
  \end{cases}
\]
Then \(\displaystyle \xi_n(\gamma_1)\) does not converge as \(n\to\infty\), so \((\xi_n)\notin f\). Take Riesz weights \(p_n=1\) for all \(n\).  Then \(P_n = \sum_{i=1}^n p_i = n\), and the Riesz (Cesàro) means are
    \[
      R_n(\xi)(\gamma_1)
      \;=\;
      \frac{1}{n} \sum_{k=1}^n \xi_k(\gamma_1)
      \;=\;
      \begin{cases}
        \tfrac12 + \tfrac{1}{2n}, & n\text{ odd},\\
        \tfrac12,                 & n\text{ even}.
      \end{cases}
    \]
    Hence 
    \(\lim_{n\to\infty} R_n(\xi)(\gamma_1) = \tfrac12\), and so \((\xi_n)\in f_R\).

Therefore \((\xi_n)\in f_R\) but \((\xi_n)\notin f\), showing that \(f_R\nsubseteq f\).
This is  the proof of the Theorem \ref{teo121}.
 \end{proof}
 
 Similar to the proof used in the theorem above, it can be proven that $e\subset e_R$, $m\subset m_R$ and $d\subset d_R$.
\begin{theorem}
If $(\xi_n(\gamma))\in e_R$ then $(\xi_n(\gamma))\in m_R$ that is the inclusion $e_R\subseteq m_R$ is valid.
 \end{theorem}
 \begin{proof}
Let us suppose that $(\xi_n(\gamma)) \in e_R$. By the definition of $e_R$, this implies that
\[
\lim_{n \to \infty}E\left[ \left|\sum_{i=1}^n \frac{p_i}{P_n} \xi_i(\gamma) - \xi(\gamma) \right| \geq \epsilon \right] = 0, \quad \text{for all } \epsilon > 0.
\]

To relate this to the expected value condition in $e_R$, observe that by the definition of expected value,
\begin{align*}
&E\left[\left|\sum_{i=1}^n \frac{p_i}{P_n} \xi_i(\gamma) - \xi(\gamma) \right|\right] = \int_{0}^{+\infty} M\left\{ \left|\sum_{i=1}^n \frac{p_i}{P_n} \xi_i(\gamma) - \xi(\gamma) \right| \geq x \right\} dx \\
&\geq \int_{0}^{\epsilon} M\left\{ \left|\sum_{i=1}^n \frac{p_i}{P_n} \xi_i(\gamma) - \xi(\gamma) \right| \geq x \right\} dx 
\geq \int_{0}^{\epsilon} M\left\{ \left|\sum_{i=1}^n \frac{p_i}{P_n} \xi_i(\gamma) - \xi(\gamma) \right| \geq \epsilon \right\} dx \\
&= \epsilon \, M\left\{ \left|\sum_{i=1}^n \frac{p_i}{P_n} \xi_i(\gamma) - \xi(\gamma) \right| \geq \epsilon \right\}.
\end{align*}

Therefore, we can write
\[
M\left\{ \left|\sum_{i=1}^n \frac{p_i}{P_n} \xi_i(\gamma) - \xi(\gamma) \right| \geq \epsilon \right\} \leq \frac{E\left[\left|\sum_{i=1}^n \frac{p_i}{P_n} \xi_i(\gamma) - \xi(\gamma) \right|\right]}{\epsilon}.
\]

Since $(\xi_n(\gamma)) \in e_R$, it follows that
\[
\lim_{n \to \infty} E\left[ \left|\sum_{i=1}^n \frac{p_i}{P_n} \xi_i(\gamma) - \xi(\gamma) \right| \geq \epsilon \right] = 0.
\]
Thus, by the above inequality, we conclude that
\[
\lim_{n \to \infty} M\left\{\left|\sum_{i=1}^n \frac{p_i}{P_n} \xi_i(\gamma) - \xi(\gamma) \right|\geq\epsilon\right\} = 0.
\]
This implies that $(\xi_n(\gamma)) \in m_R$, and hence $e_R \subseteq m_R$.
This completes the proof.\end{proof}

 As is well known, the Riesz matrix is a regular summability method. Considering the sets \( f_R \), \( m_R \), and \( e_R \) obtained using the domain of this matrix, it can be easily seen that \( f \subseteq f_R \),  \( m \subseteq m_R \) and \( e \subseteq e_R \).  These statements imply that it is possible to create new sets of uncertain variables in the domain of the given matrix between the sets \( e \) and \( m \), and between the sets \( m \) and \( f \), with an infinite number of uncertain variables.

Let's state the following proposition without proof:
\begin{prop} Let $A$ be a regular summability method of real or complex numbers. Then the inclusions  \((\xi_n(\gamma))\in e\Rightarrow (\xi_n(\gamma))\in e_A$,  $(\xi_n(\gamma))\in m\Rightarrow (\xi_n(\gamma))\in m_A$ and $ (\xi_n(\gamma))\in f\Rightarrow (\xi_n(\gamma))\in f_A\) are holds.
\end{prop}

With this way,  new sets of uncertain variables can be constructed from the sets 
$f$, $m$ and $e$.

  \begin{theorem}\label{iso}
The sequence spaces of uncertain variables $f_R$, $m_R$ and $e_R$ are linearly isomorphic to the
spaces $f$, $m$ and $e$, respectively.
\end{theorem}
\begin{proof}
 In order to prove the fact $f_R\cong f$ , we should show the existence of a linear
bijection between  the spaces $f_R$ and $f$. Consider the transformation
$T$ defined, with the notation of (\ref{d4}), from $f_R$ to
$f$ by $\xi(\gamma)\mapsto T\xi(\gamma)=\eta(\gamma). $ The linearity of $T$ is clear. Further,
it is trivial that $\xi(\gamma)=\theta(\gamma)=(0,0,...)$ whenever $T\xi(\gamma)=\theta(\gamma)$ and
hence $T$ is injective. It remains to show that T is surjective. Since the matrix $R$ is a lower triangular matrix, it has an $R^{-1}$ inverse.  It is clear that for every $(\eta_n(\gamma)) \in f$, there is an element $(\xi_n(\gamma)) \in f_R$ such that $\xi_n(\gamma)=R^{-1}\eta_n(\gamma)$. That is  $T$ is onto from $f_R$ to $f$. The $m_R \cong m$ and $e_R \cong e$ can be proved in a similar way, so we omit them.
\end{proof}

  Let $(\xi_n(\gamma))$ be a sequence uncertain variables and  define the set $\tilde{m}$ as follows: {\small \begin{equation}\tilde{m}=\left\{(\xi_n(\gamma)):M\left(\bigcap_{m=1}^{\infty}\bigcup_{n=m}^{\infty}\{\gamma\in\Gamma:|\sum_{i=1}^n\frac{p_i}{P_n}\xi_i(\gamma)-\xi(\gamma)|\geq\epsilon\}\right)=0, ~~\text{for all}~~\epsilon>0\right\}.
  \end{equation}}
 \begin{prop}\label{pro1}
 The condition  $(\xi_n(\gamma))\in f_R$ if and only if $(\xi_n(\gamma))\in \tilde{m}$ is holds.
 \end{prop}
 \begin{proof}
The proof is easily obtained from definition of Riesz  type convergent almost sure in (\ref{ms1}).
 \end{proof}

\begin{theorem}\label{teo1}
 Let consider  the sequence $(\xi_n(\gamma))$ of uncertain variables. Then   $(\xi_n(\gamma))\in u_R$  if and only if $(\xi_n(\gamma))\in \tilde{u}_R$.
\end{theorem}
\begin{proof}
Firstly we show that  if $(\xi_n(\gamma)) \in u_R$, then $(\xi_n(\gamma)) \in \tilde{u}_R$. Now let suppose that $(\xi_n(\gamma)) \in u_R$. By definition 
there exists a sequence of sets $(E_k) \subset \Gamma$ with $M(E_k) \to 0$ as $k \to \infty$, such that $(\xi_n(\gamma))$ converges uniformly to $\xi(\gamma)$ on $\Gamma \setminus E_k$ for any fixed $k$.
To prove that $(\xi_n(\gamma)) \in \tilde{u}_R$, we need to show:
\[
\lim_{m \to \infty} M \left( \bigcup_{n=m}^{\infty} \left\{ \gamma \in \Gamma : \left| \sum_{i=1}^n \frac{p_i}{P_n} \xi_i(\gamma) - \xi(\gamma) \right| \geq \epsilon \right\} \right) = 0
\]
for any $\epsilon > 0$.

Since $(\xi_n(\gamma))$ converges uniformly to $\xi(\gamma)$ on $\Gamma \setminus E_k$ and $M(E_k) \to 0$, this implies that for any given $\epsilon > 0$, there exists an integer $N$ (depending on $\epsilon$ and $k$) such that for all $n \geq N$, we have
\[
\left| \sum_{i=1}^n \frac{p_i}{P_n} \xi_i(\gamma) - \xi(\gamma) \right| < \epsilon \quad \text{for all } \gamma \in \Gamma \setminus E_k.
\]
Consequently, for sufficiently large $m$, the union 
\[
\bigcup_{n=m}^{\infty} \left\{ \gamma \in \Gamma : \left| \sum_{i=1}^n \frac{p_i}{P_n} \xi_i(\gamma) - \xi(\gamma) \right| \geq \epsilon \right\}
\]
will only intersect with $E_k$, which has measure going to zero as $k \to \infty$. Therefore, we obtain
\[
\lim_{m \to \infty} M \left( \bigcup_{n=m}^{\infty} \left\{ \gamma \in \Gamma : \left| \sum_{i=1}^n \frac{p_i}{P_n} \xi_i(\gamma) - \xi(\gamma) \right| \geq \epsilon \right\} \right) = 0.
\]
Thus, $(\xi_n(\gamma)) \in \tilde{u}_R$.

Conversely, we will show that if $(\xi_n(\gamma)) \in \tilde{u}_R$, then $(\xi_n(\gamma)) \in u_R$. Now assume that $(\xi_n(\gamma)) \in \tilde{u}_R$. This implies that for any $\epsilon > 0$,
\[
\lim_{m \to \infty} M \left( \bigcup_{n=m}^{\infty} \left\{ \gamma \in \Gamma : \left| \sum_{i=1}^n \frac{p_i}{P_n} \xi_i(\gamma) - \xi(\gamma) \right| \geq \epsilon \right\} \right) = 0.
\]

To show that $(\xi_n(\gamma)) \in u_R$, we need to find sets $(E_k)$ with $M(E_k) \to 0$ as $k \to \infty$ such that $(\xi_n(\gamma))$ converges uniformly to $\xi(\gamma)$ on $\Gamma \setminus E_k$ for any fixed $k$.

For a fixed $k$, set
\[
E_k = \bigcup_{n=k}^{\infty} \left\{ \gamma \in \Gamma : \left| \sum_{i=1}^n \frac{p_i}{P_n} \xi_i(\gamma) - \xi(\gamma) \right| \geq \frac{1}{k} \right\}.
\]
By assumption, $M(E_k) \to 0$ as $k \to \infty$.

On $\Gamma \setminus E_k$, for any $n \geq k$, we have
\[
\left| \sum_{i=1}^n \frac{p_i}{P_n} \xi_i(\gamma) - \xi(\gamma) \right| < \frac{1}{k}.
\]
This implies that $(\xi_n(\gamma))$ converges uniformly to $\xi(\gamma)$ on $\Gamma \setminus E_k$ as required by the definition of $u_R$. Therefore, $(\xi_n(\gamma)) \in u_R$.

Consequently, since we have shown both directions, we conclude that \(
u_R = \tilde{u}_R.
\) Thus  the Theorem \ref{teo1} is proved.
\end{proof}

\begin{theorem}
 The sequence $(\xi_n(\gamma))$ of uncertain variables uniformly Riesz type convergent in measure to $\xi(\gamma)$ if and only if there exists subsequence $(\xi_{n'_k}(\gamma))$ of $(\xi_{n'}(\gamma))$ such that $(\xi_{n'_k}(\gamma))$ is Riesz  type uniformly converges uniformly almost sure to $\xi(\gamma)$, $(k\rightarrow\infty)$, for any subsequence $(\xi_{n'}(\gamma))$ of $(\xi_n(\gamma))$.
\end{theorem}
\begin{proof}
Suppose that  the sequence $(\xi_n(\gamma))$ of uncertain variables Riesz type convergent in measure to $\xi(\gamma)$. Then the sequence $(\xi_{n'}(\gamma))$ of uncertain variables  Riesz type convergent in measure to $\xi(\gamma)$. It follows  from the definition of Riesz type converges in measure that there exists a subsequence  $(\xi_{n'_k}(\gamma))$ of $(\xi_{n'}(\gamma))$  such that
\eqy {M\{\gamma\in\Gamma:|\sum_{i=1}^n\frac{p_i}{P_n}\xi_i(\gamma)-\xi(\gamma)|\geq\frac{1}{k}\}\leq\frac{1}{2^k}~~\text{for any}~~k\geq 1.} From here we see that
\eqy {M(\bigcup_{k=m}^{\infty}\{\gamma\in\Gamma:|\sum_{i=1}^n\frac{p_i}{P_n}\xi_i(\gamma)-\xi(\gamma)|\geq\frac{1}{k}\})\leq\sum_{k=m}^{\infty}\frac{1}{2^k}
=\frac{1}{2^{m-1}}} for any $m\geq1.$ Thus
\eqy{\lim_mM(\bigcup_{k=m}^{\infty}\{\gamma\in\Gamma:|\sum_{i=1}^n\frac{p_i}{P_n}\xi_i(\gamma)-\xi(\gamma)|\geq\frac{1}{k}\})=0}
This means that  $(\xi_{n'_k}(\gamma))$ is  Riesz type converges uniformly almost sure convergent  $\xi(\gamma)$.

Now let us suppose that the sequence $(\xi_n(\gamma))$ of uncertain variables  does not  Riesz type convergent in measure to $\xi(\gamma)$. Then there exists a $\epsilon>0$ such that
\eqy{\lim_mM(\bigcup_{k=m}^{\infty}\{\gamma\in\Gamma:|\sum_{i=1}^n\frac{p_i}{P_n}\xi_i(\gamma)-\xi(\gamma)|\geq\frac{1}{k}\})>K>0.}
Then there exists a subsequence $(\xi_{n'}(\gamma))$ of $(\xi_{n}(\gamma))$  such that \eqy{M(\bigcup_{k=m}^{\infty}\{\gamma\in\Gamma:|\sum_{i=1}^{n'}\frac{p_i}{P_n}\xi_i(\gamma)-\xi(\gamma)|\geq\epsilon\})>K.}
It is clear that $(\xi_{n'}(\gamma))$ has no subsequence Riesz type converges uniformly almost sure convergent to $\xi(\gamma)$ then $(\xi_n(\gamma))$  Riesz type convergent in measure to $\xi(\gamma)$.
\end{proof}

The following theorem is describe some important propositionerties of Riesz type  convergent in measure and Riesz type converges in  $p-$distance.

\begin{theorem}
If uncertain variables sequence $(\xi_n(\gamma))$ Riesz type convergent to $\xi(\gamma)$ and $\eta(\gamma)$, respectively then $M\{\xi(\gamma)=\eta(\gamma)\}=1.$
\end{theorem}
\begin{proof}
If we consider subadditivity axiom then we have
 \begin{align*}
 M\{|\xi(\gamma)-\eta(\gamma)&+\sum_{i=1}^n\frac{p_i}{P_n}\xi_i(\gamma)-\sum_{i=1}^n\frac{p_i}{P_n}\xi_i(\gamma)|\geq\epsilon\}\\&\leq M\{|\sum_{i=1}^n\frac{p_i}{P_n}\xi_i(\gamma)-\xi(\gamma)|\geq\frac{\epsilon}{2}\}+ M\{|\sum_{i=1}^n\frac{p_i}{P_n}\xi_i(\gamma)-\eta(\gamma)|\geq\frac{\epsilon}{2}\}\\&\rightarrow 0~~\text{for}~~ n\rightarrow\infty.
 \end{align*} 
\end{proof}
\section{Conclusion}
This study has thoroughly examined the Riesz-type convergence properties of uncertain variables and their relationships in functional analysis topics. The findings have clarified the connections between the convergence types of uncertain variables in sequence spaces and given the mathematical analysis of these convergence types. One of the key results of the study is its clarification of the differences between the Riesz-type convergence of uncertain variables and other convergence types, emphasizing the significant role of Riesz-type convergence within the framework of functional analysis. Additionally, determining the isomorphisms of uncertain variables in various sequence spaces has contributed to a deeper understanding of the transitions between such spaces. The inclusions was given as table in the Table \ref{tt1}.

It is believed that the Riesz-type convergence of uncertain variables holds a significant place in mathematical modeling and applied fields. The mathematical framework presented in this study is expected to inspire more in-depth research on uncertainty theory and convergence analysis.

\begin{table}[H]
\centering
\caption{Summary of the inclusions as a table.}
\label{tt1}
\begin{tabular}{|c|c|c|c|c|c|c|}
\hline
$(\xi_n(\gamma))\in f$ & \cellcolor{black!50} & $(\xi_n(\gamma))\in e$ & $\Longrightarrow$ & $(\xi_n(\gamma))\in m$ & $\Longrightarrow$ & $(\xi_n(\gamma))\in d$ \\ \hline
$\Downarrow$ &  \cellcolor{black!50} & $\Downarrow$ &  \cellcolor{black!50} & $\Downarrow$ &  \cellcolor{black!50} & $\Downarrow$ \\ \hline
$(\xi_n(\gamma))\in f_R$ & \cellcolor{black!50} & $(\xi_n(\gamma))\in e_R$ & $\Longrightarrow$ & $(\xi_n(\gamma))\in m_R$ & $\Longrightarrow$ & $(\xi_n(\gamma))\in d_R$ \\ \hline
\multicolumn{7}{|c|}{\( m_R \cap f_R \neq \emptyset  \)} \\ \hline
\end{tabular}
\end{table}
\subsection{Conjecture}
The question of whether the spaces $f_R$ and $e_R$, or $f_R$ and $d_R$, overlap remains open, as research on this issue is still ongoing.

\end{document}